\documentclass[preprint,11pt]{elsarticle}

\usepackage{amsfonts, amsmath, amscd}
\usepackage[psamsfonts]{amssymb}

\usepackage{amssymb}

\usepackage{pb-diagram}

\usepackage[all,cmtip]{xy}

\usepackage[usenames]{color}

\newtheorem{theorem}{Theorem}[section]
\newtheorem{lemma}[theorem]{Lemma}
\newtheorem{corollary}[theorem]{Corollary}

\newtheorem{proposition}[theorem]{Proposition}
\newtheorem{definition}[theorem]{Definition}
\newtheorem{example}[theorem]{Example}

\newtheorem{problem}[theorem]{Problem}

\newproof{proof}{Proof}

\numberwithin{equation}{section}

\newcommand{\e}{\varepsilon}
\newcommand{\w}{\omega}

\newcommand{\IR}{\mathbb{R}}

\newcommand{\cl}{\mathrm{cl}}

\newcommand{\CC}{C_k}

\newcommand{\SM}{{\setminus}}

\input xy
\xyoption{all}

\begin{document}

\begin{frontmatter}

\title{Ascoli's theorem for  pseudocompact  spaces}

\author{S.~Gabriyelyan}
\ead{saak@math.bgu.ac.il}
\address{Department of Mathematics, Ben-Gurion University of the Negev, Beer-Sheva, P.O. 653, Israel}

\begin{abstract}
A Tychonoff space $X$ is called ({\em sequentially}) {\em Ascoli} if every compact subset (resp. convergent sequence) of $\CC(X)$ is equicontinuous, where $\CC(X)$ denotes the space of all real-valued continuous functions on $X$ endowed with the compact-open topology. The classical Ascoli theorem states that each compact space is Ascoli. We show that a pseudocompact space $X$ is  Asoli iff it is sequentially Ascoli iff  it is selectively  $\w$-bounded.
\end{abstract}


\end{frontmatter}



\section{Introduction}


All topological spaces in the article are assumed to be Tychonoff. We denote by $\CC(X)$ the space $C(X)$ of all continuous real-valued functions on a space $X$ endowed with the compact-open topology. One of the basic theorems in Analysis is the Ascoli theorem which states that {\em if $X$ is a $k$-space, then every compact subset of $\CC(X)$ is equicontinuous.}
For the proof of the Ascoli theorem and various its applications see for example the classical books \cite{Edwards}, \cite{Eng} or \cite{NaB}. 
The Ascoli theorem motivates us in \cite{BG} to introduce and study the class of Ascoli spaces. A space $X$ is called {\em Ascoli} if every compact subset of $\CC(X)$ is equicontinuous.
In \cite{Noble}, Noble proved that every $k_\IR$-space is Ascoli (recall that a  space $X$ is called a {\em $k_\IR$-space} if a real-valued function $f$ on $X$ is continuous if and only if its restriction $f{\restriction}_K$ to any compact subset $K$ of $X$ is continuous). However, there are Ascoli spaces which are not $k_\IR$, see \cite{BG}.
Being motivated by the classical notion of $c_0$-barrelled locally convex spaces and the fact that in many highly important cases in Analysis only convergent sequences are considered (as in the Lebesgue Dominated Convergence Theorem), we defined in \cite{Gabr:weak-bar-L(X)}  a space $X$ to be {\em sequentially Ascoli} if every convergent sequence in $\CC(X)$ is equicontinuous. Clearly, every Ascoli space is sequentially Ascoli, but the converse is not true in general (every non-discrete $P$-space is sequentially Ascoli but not Ascoli, see \cite{Gabr:weak-bar-L(X)}). Ascoli and sequentially Ascoli spaces in various classes of topological, function and locally convex spaces are thoroughly studied in
\cite{Banakh-Survey,BG,Gabr-LCS-Ascoli,Gab-LF,Gabr-reflex-L(X),Gabr-seq-Ascoli,GGKZ,GGKZ-2,GKP,Noble}. The next diagram shows the relationships between the aforementioned classes of topological spaces
\[
\xymatrix{
\mbox{$k$-space}  \ar@{=>}[r] & \mbox{$k_\IR$-space} \ar@{=>}[r] & \mbox{Ascoli} \ar@{=>}[r] & \mbox{sequentially Ascoli}.}
\]

By the Ascoli theorem every compact space is Ascoli. Although the compact spaces form the most important class of topological spaces, there are other classes of compact-type topological spaces (as sequentially compact or countably compact spaces etc.) which play a considerable role both in Analysis and General Topology, see for example  \cite{Eng}, \cite{GiJ}, \cite{NaB} or the articles \cite{MW,ST}. The most general class of compact-type spaces is the class of pseudocompact spaces. Recall that a space $X$ is called {\em pseudocompact} if every continuous function on $X$ is bounded. So the following question arises naturally: {\em Which pseudocompact spaces $X$ are} ({\em sequentially}) {\em Ascoli}? A partial answer to this question was obtained in \cite{Gabr-seq-Ascoli} where we showed that totally countably compact spaces and near sequentially compact spaces are sequentially Ascoli, however, there are countably compact spaces which are not sequentially Ascoli (for definitions see Section \ref{sec:swb-main}).

Let $X$ be a  pseudocompact space. We denote by $\beta X$ the Stone-\v{C}ech compactification of $X$, and let $\beta: X\to \beta X$ be the canonical embedding. Then the adjoint (or restriction)  map $\beta^\ast: C(\beta X)\to \CC(X)$, $\beta^\ast(f)=f\circ \beta$, is a continuous linear isomorphism from the Banach space $C(\beta X)$ onto $\CC(X)$. One of the most important properties of continuous functions is the property of being compact-covering. A continuous function $f:X\to Y$ between topological space $X$ and $Y$ is called {\em compact-covering} if for every compact subset $K$ of $Y$ there is a compact subset $C$ of $X$ such that $f(C)=K$. It is well known that perfect mappings are compact-covering (\cite[Theorem~3.7.2]{Eng}), and compact-covering functions are important for the study of functions spaces as $\CC(X)$, see \cite{mcoy}. Therefore one can ask: {\em For which pseudocompact spaces $X$ the adjoint map $\beta^\ast: C(\beta X)\to \CC(X)$ is compact-covering}?


The following class of pseudocompact spaces plays a crucial role to answer both questions.
\begin{definition} {\em
A space $X$ is called {\em selectively  $\w$-bounded } if for any sequence $\{U_n\}_{n\in\w}$ of nonempty open subsets of $X$ there exists a sequence $(x_n)_{n\in\w}\in\prod_{n\in\w}U_n$ containing a subsequence $(x_{n_k})_{k\in\w}$ with compact closure.\qed}
\end{definition}
Our terminology is explained by the possibility to ``select'' special (sub)sequences and the classical notion of $\w$-bounded spaces (recall that a space $X$ is {\em $\w$-bounded} if every sequence in $X$ has compact closure). Clearly, $\w$-bounded (in particular, compact) spaces and sequentially compact spaces are selectively  $\omega$-bounded, and every selectively  $\w$-bounded space is pseudocompact. In Lemma \ref{l:swb-Frolik} below we show that the class of selectively  $\w$-bounded spaces coincides with Frol\'{\i}k's class $\mathfrak{P}^\ast$ introduced in \cite{Frolik-tcc}. In the next section we also compare the class of selectively $\w$-bounded spaces with other important classes of pseudocompact spaces. 


The following theorem proved in the next section is the main result of the paper.

\begin{theorem} \label{t:seq-Ascoli-mu-seq}
For a pseudocompact space $X$ the following assertions are equivalent:
\begin{enumerate}
\item[{\rm(i)}] $X$ is selectively  $\w$-bounded;
\item[{\rm(ii)}] the adjoint map $\beta^\ast : C(\beta X)\to \CC(X)$ is compact-covering;
\item[{\rm(iii)}] $X$ is an  Asoli space;
\item[{\rm(iv)}] $X$ is a sequentially  Asoli space.
\end{enumerate}
\end{theorem}

It is worth mentioning that Kato constructed in \cite{Kato} a space $X$ in the class $\mathfrak{P}^\ast$ which is not a $k_\IR$-space. This example, Lemma \ref{l:swb-Frolik} and Theorem \ref{t:seq-Ascoli-mu-seq} show that there is even a pseudocompact Ascoli space which is not a $k_\IR$-space.



\section{Proof of Theorem \ref{t:seq-Ascoli-mu-seq}} \label{sec:swb-main}


In \cite{Frolik-tcc}, Frol\'{\i}k defined the class $\mathfrak{P}^\ast$ consisting of spaces with the property: each infinite collection of {\em disjoint} open sets has an infinite subcollection each of which meets some fixed compact set. Below we show that the class $\mathfrak{P}^\ast$  coincides with the class of all selectively  $\w$-bounded spaces.

\begin{lemma} \label{l:swb-Frolik}
A space $X$ belongs to $\mathfrak{P}^\ast$ if and only if it is selectively  $\w$-bounded.
\end{lemma}

\begin{proof}
It is clear that every selectively  $\w$-bounded space belongs to $\mathfrak{P}^\ast$. Conversely, assume that $X$ belongs to $\mathfrak{P}^\ast$. Fix a sequence $\{U_n\}_{n\in\w}$ of nonempty open subsets of $X$. We have to show that there exists a sequence $(x_n)_{n\in\w}\in\prod_{n\in\w}U_n$ containing a subsequence $(x_{n_k})_{k\in\w}$ with compact closure. Set $A:=\{n\in\w: U_n \mbox{ is infinite}\}$ and consider two cases.
\smallskip

{\em Case 1. The family $A$ is finite.} Then, without loss of generality, we can assume that $U_n=\{x_n\}$ for every $n\in\w$. Put $S:=\{x_n\}_{n\in\w}$. If $S$ is finite, then the sequence $(x_n)\in \prod_{n\in\w} U_n$ has compact closure. If $S$ is infinite, take a subsequence $\{x_{n_k}\}_{k\in\w}$ of $S$ consisting of pairwise distinct points. Since $X\in \mathfrak{P}^\ast$ there is a compact set $K$ such that the set $J:=\{ j\in\w: U_{n_j}\cap K\not=\emptyset\}$ is infinite. Then the subsequence $(x_{n_j})_{j\in J}$ of $(x_n)_{n\in\w}$ has compact closure.
\smallskip

{\em Case 2. The family $A$ is infinite.} Then, passing to a subsequence if needed, we can assume that all $U_n$ are infinite. By induction on $n\in\w$, we can choose pairwise distinct points $z_n$ such that $z_n\in U_n$ for every $n\in\w$. Once again by induction, one can choose a subsequence $\{z_{n_k}\}_{k\in\w}$ of $\{z_n\}_{n\in\w}$ and a sequence $\{V_k\}_{k\in\w}$ of open sets in $X$ such that $z_{n_k}\in V_k\subseteq U_{n_k}$ and $V_k\cap V_j=\emptyset$ for all distinct $k,j\in\w$. Since $X\in \mathfrak{P}^\ast$ and all $V_k$ are pairwise disjoint, there is a compact set $K$ such that the set $J:=\{k\in\w: K\cap V_k\not= \emptyset\}$ is infinite. For every $j\in J$, choose a point $x_{n_j}\in K\cap V_j$ and, for every $n\not\in \{n_j:j\in J\}$, let $x_n := z_n$. It is clear that $(x_n)_{n\in\w}\in\prod_{n\in\w}U_n$ and its subsequence $(x_{n_j})_{j\in J}$ has compact closure witnessing the property of being a selectively $\w$-bounded space.\qed
\end{proof}

Now we compare the class of selectively $\w$-bounded spaces with other important classes of  pseudocompact spaces. We recall that a space $X$ is called
\begin{enumerate}
\item[$\bullet$] {\em sequentially compact} if every sequence in $X$ has a convergent subsequence;
\item[$\bullet$] {\em totally countably compact} if every sequence in $X$ has a subsequence with compact closure;
\item[$\bullet$]  {\em near sequentially compact } if for any sequence $\{U_n\}_{n\in\w}$ of open subsets of $X$ there exists a sequence $(x_n)_{n\in\w}\in\prod_{n\in\w}U_n$ containing a convergent subsequence $(x_{n_k})_{k\in\w}$;
\item[$\bullet$] {\em countably compact} if every sequence in $X$ has a cluster point.
\end{enumerate}
Near sequentially compact spaces were introduced and  studied by Dorantes-Aldama and Shakhmatov \cite{DAS1}, who called them selectively sequentially pseudocompact spaces. Later those spaces were applied in  \cite{BG-JNP} to the study of the Josefson--Nissenzweig property in the realm of locally convex spaces.
Totally countably compact spaces, introduced by Frol\'{\i}k  \cite{Frolik-tcc}, were intensively studied by Vaughan in \cite{Vaughan}.
Evidently, totally countably compact spaces and  near sequentially compact spaces are selectively  $\w$-bounded.

By Example~2.6 of \cite{DAS1}, the  Mr\'{o}wka--Isbell space associated with a maximal almost disjoint family $\mathcal{A}$ on the discrete space $\w$ is near sequentially compact. This example and Examples~2.11 and 2.14 from \cite{Vaughan} show that none of the implications in the following diagram  is in general reversible
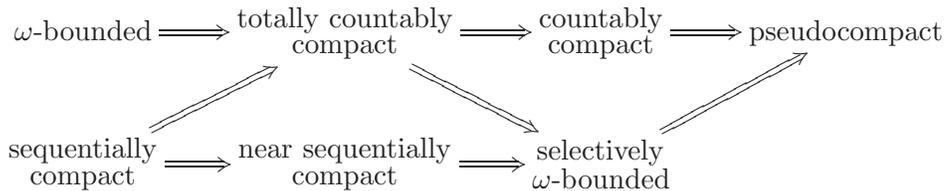
\[
\xymatrix{
\mbox{$\w$-bounded}  \ar@{=>}[r] & {\substack{\mbox{totally countably} \\ \mbox{compact}}}  \ar@{=>}[r] \ar@{=>}[rd] & {\substack{\mbox{countably} \\ \mbox{compact}}}  \ar@{=>}[r] & \mbox{pseudocompact}\\
{\substack{\mbox{sequentially} \\ \mbox{compact}}}  \ar@{=>}[r]  \ar@{=>}[ru] &  {\substack{\mbox{near sequentially} \\ \mbox{compact}}} \ar@{=>}[r]  & {\substack{\mbox{selectively} \\ \mbox{$\w$-bounded}}} \ar@{=>}[ru] &
}
\]
Moreover, in \cite{BG-pseudo} we constructed an example of a selectively $\w$-bounded space $X$ which is countably compact but not totally countably compact.

 Let $X$ be a (Tychonoff) space. Then the  sets of the form
\[
[K;\e]:=\{ f\in C(X): |f(x)|<\e \; \mbox{ for all }\; x\in K\}, 
\]
where $K\subseteq X$ is compact and $\e>0$, form a base of the compact-open topology on $C(X)$. The space $C(X)$ endowed with the pointwise topology is denoted by $C_p(X)$.

Recall that a continuous function $f:X\to Y$ between topological spaces $X$ and $Y$ is called {\em sequence-covering} if for every convergent sequence $S$  in $Y$ (with the limit point) there is a convergent sequence $C\subseteq X$ such that $f(C)=K$.

\begin{proposition} \label{p:mapping-Ascoli}
Let $X$ be a subspace of a space $Y$ such that the adjoint map $i^\ast :\CC(Y)\to \CC(X)$ of the identical embedding $i:X\hookrightarrow Y$ is surjective. If $i^\ast$ is compact (sequence) covering and $Y$ is a (resp. sequentially) Ascoli space, then so is $X$.
\end{proposition}

\begin{proof}
Let $K$ be a compact subset (or a convergent sequence) in $\CC(X)$. We have to show that $K$ is equicontinuous. Fix a point $x_0\in X$ and $\e>0$. Choose a compact subset (or a convergent sequence) $C$  in $\CC(Y)$ such that $i^\ast(C)=K$. Since $Y$ is  (sequentially) Ascoli, there is an open neighborhood $U$ of $x_0$ in $Y$ such that
\begin{equation} \label{equ:mapping-Ascoli-1}
|g(y)-g(x_0)|<\e \;\mbox{ for all }\; y\in U \mbox{ and } g\in C.
\end{equation}
For every $x\in U\cap X$ and each $f\in K$, take $g\in C$ such that $f=g\circ i$ and then (\ref{equ:mapping-Ascoli-1}) implies
$
|f(x)-f(x_0)|=\big|g\big(i(x)\big)-g\big(i(x_0)\big)\big| <\e.
$
Thus $K$ is equicontinuous.\qed
\end{proof}

Now we are able to prove our main result.
\smallskip

{\em Proof of Theorem \ref{t:seq-Ascoli-mu-seq}}.  
(i)$\Rightarrow$(ii) Assume that $X$ is selectively  $\w$-bounded, and let $K$ be a compact subset of $\CC(X)$. We have to show that the closed subset $C:=(\beta^\ast)^{-1}(K)$ of the Banach space $C(\beta X)$ is compact. Suppose for a contradiction that $C$ is not compact. Since $C(\beta X)$ is complete and $C$ is closed, it follows that $C$ is not precompact in $C(\beta X)$. Therefore, by \cite[Theorem~5]{BGP}, there exist a sequence $\{ f_n\}_{n\in\w}\subseteq C$ and $\e>0$ such that
\begin{equation} \label{equ:mapping-Ascoli-2}
\| f_n -f_m\|_{\infty} >\e \quad \mbox{ for all distinct }\; n,m\in\w,
\end{equation}
where $\|f\|_\infty$ denotes the sup-norm of $f\in C(\beta X)$.

It is clear that $K$ is compact also in the space $C_p(X)$. Since $X$ is pseudocompact, Theorem~III.4.22 of \cite{Arhangel} implies that $K$ is an Eberlein compact, and hence $K$ is Fr\'{e}chet--Urysohn by \cite[Theorem~III.3.6]{Arhangel}. Therefore, passing to a subsequence if needed, we can assume that the sequence $\{ f_n\}_{n\in\w}$ converges  in $\CC(X)$ to some function $g\in K$. Replacing $K$ by $K-g$, we can also suppose that $g=\mathbf{0}$ is the zero function.

Since $X$ is dense in $\beta X$, (\ref{equ:mapping-Ascoli-2}) implies that for every $n\in \w$, the open set
\[
U_n :=\{ x\in X: |f_n(x) -f_{n+1}(x)|>\e\}
\]
is not empty. As $X$ is selectively  $\w$-bounded, there exists a sequence $(x_n)_{n\in\w}\in\prod_{n\in\w}U_n$ containing a subsequence $(x_{n_k})_{k\in\w}$ whose closure $S:=\overline{\{x_{n_k}: k\in\w\}}$ is a compact subset of $X$.

Now, since $f_n \to \mathbf{0}$ in $\CC(X)$, there is an $m\in\w$ such that $f_n\in \big[S;\tfrac{\e}{3}\big]$ for all $n\geq m$. In particular, we have
\begin{equation} \label{equ:mapping-Ascoli-3}
\big| f_{n_k}(x_{n_k}) -f_{n_k+1}(x_{n_k})\big| \leq \big| f_{n_k}(x_{n_k})\big|+\big| f_{n_k+1}(x_{n_k})\big| < \tfrac{2\e}{3}
\end{equation}
for all sufficiently large $k\in \w$. But since $x_{n_k}\in U_{n_k}$ for all $k\in\w$, (\ref{equ:mapping-Ascoli-3}) contradicts to the choice of the open sets $U_n$. This contradiction shows that $C$ is compact in $C(\beta X)$, and hence the map $\beta^\ast$ is compact-covering.
\smallskip

The implication (ii)$\Rightarrow$(iii) follows from  Proposition \ref{p:mapping-Ascoli} applied to $X$ and $Y=\beta X$, and the implication (iii)$\Rightarrow$(iv) is trivial.
\smallskip

(iv)$\Rightarrow$(i) Assume that $X$ is a sequentially  Asoli space. We have to show that $X$ is  selectively  $\w$-bounded. Suppose for a contradiction that $X$ is not a selectively  $\w$-bounded space. Then there exists a sequence $\{U_n\}_{n\in\w}$ of nonempty open subsets of $X$ such that for every sequence $(z_n)_{n\in\w}\in\prod_{n\in\w}U_n$ there is no subsequence $(z_{n_k})_{k\in\w}$ whose closure is compact. 

For every $n\in\w$, choose a point $x_n\in U_n$ and a continuous function $f_n:X\to [0,1]$ such that $f_n(x_n)=1$ and $f_n(X\SM U_n)\subseteq \{0\}$. We claim that $f_n\to \mathbf{0}$ in $\CC(X)$. Indeed, fix a compact subset $K$ of $X$ and $\e>0$. Then the choice of the sequence  $\{U_n\}_{n\in\w}$  implies that the set $A:=\{n\in\w: U_n\cap K\not=\emptyset\}$ is finite (indeed, otherwise, we could choose a point $z_n\in U_n\cap K$ for every $n\in A$ and an arbitrary point $z_n\in U_n$ for every $n\in\w\SM A$, and then the closure of the subsequence $\{z_n : n\in A\}$ of $\{z_n\}_{n\in\w}$ would be compact that contradicts the choice of the sequence $\{U_n\}_{n\in\w}$). This means that $f_n\in [K;\e]$ for every $n\in \w\SM A$. Thus $f_n\to \mathbf{0}$. Set $S:=\{f_n\}_{n\in\w} \cup\{ \mathbf{0}\}$, so $S$ is a convergent sequence in $\CC(X)$.

For every $n\in\w$, set $V_n:=\{ x\in X: f_n(x)>\tfrac{1}{2}\}$; so $x_n\in V_n \subseteq U_n$. Since $X$ is pseudocompact, the family $\{V_n\}_{n\in\w}$ is not locally finite (see \cite[Theorem~3.10.22]{Eng}), and therefore there is a point $z\in X$ such that for every neighborhood $W$ of $z$, the set $\{n\in\w: V_n\cap W\not=\emptyset\}$ is infinite. 

Finally, to get a desired contradiction we show that the sequence $S$ is not equicontinuous. Since $X$ is sequentially Ascoli, Theorem 2.7 of \cite{Gabr-seq-Ascoli} states that $S$ is  equicontinuous if and only if $S$ is  evenly continuous, i.e. the evaluation map $\psi:S\times X\to \IR$, $\psi(f,x):=f(x)$, is  continuous (see also Lemma 2.1 of \cite{Gabr-LCS-Ascoli}). Therefore it is sufficient to show that the map $\psi$  is {\em not} continuous  at the point $(\mathbf{0},z)$. To this end, fix a $k\in\w$ and an open neighborhood $W$ of the point $z$. Since  the set $\{n\in\w: V_n\cap W\not=\emptyset\}$ is infinite,  there is an $m>k$ such that $V_m \cap W$ contains some point $t_m$. By the definition of $V_m$ we obtain $|\psi(f_m,t_m)-\psi(\mathbf{0},z)|=f_m(t_m)> \tfrac{1}{2}$. Thus $\psi$ is not continuous  at $(\mathbf{0},z)$.\qed
\smallskip

It immediately follows from Theorem \ref{t:seq-Ascoli-mu-seq} that if $X$ is a selectively  $\w$-bounded space, then every compact subset of $\CC(X)$ is metrizable. However, the converse is not true in general as the following example shows.

\begin{example} \label{exa:swb-non-swb} 
There is a countably compact non-selectively  $\w$-bounded space $X$ such that all compact subsets of $\CC(X)$ (even of $C_p(X)$) are metrizable. 
\end{example}

\begin{proof}
In \cite{Terasaka}, Terasaka constructed a separable countably compact space $X$ whose square $X\times X$ is not pseudocompact. By Theorem 3.5 of \cite{Frolik-tcc}, the product of a selectively $\w$-bounded space and a pseudocompact space is pseudocompact. Therefore the space $X$ is not selectively  $\w$-bounded. 
Let $D$ be a countable dense subspace of $X$. Then the restriction map $C_p(X)\to C_p(D)$ is continuous and injective. Since $C_p(D)$ is a metric space, it follows that all compact subsets of $C_p(X)$ and hence of $\CC(X)$ are metrizable.\qed
\end{proof}


\medskip
{\bf Acknowledge:} The author thanks Taras Banakh for useful discussion on the name ``selectively  $\w$-bounded''.

\medskip


\bibliographystyle{amsplain}

\end{document}